\documentclass[review]{elsarticle}
\usepackage[T1]{fontenc}  	
\usepackage[latin1]{inputenc}
\usepackage[french, english]{babel} 		
\usepackage{amssymb}  			
\usepackage{amsmath}	
\usepackage{amsthm}	
\usepackage{units}
\usepackage{esint}  				
\usepackage{bbm}   				
\usepackage{color,graphicx} 			
\usepackage[small]{caption}
\usepackage{cases}	
\usepackage{mathpazo}

\newcommand{\be}{\begin{equation}}
\newcommand{\ee}{\end{equation}}

\renewcommand{\Re}{\mathop{\rm Re}\nolimits}
\DeclareMathOperator{\sgn}{sgn} 
\DeclareMathOperator{\li}{li}

\newtheorem{theorem}{Theorem}
\newtheorem{lemma}{Lemma}
\newtheorem{proposition}{Proposition}
\newtheorem{corollary}{Corollary}
\newtheorem{remark}{Remark}

\makeatletter
\newcommand{\specialnumber}[1]{
\def\tagform@##1{\maketag@@@{(\ignorespaces##1\unskip\@@italiccorr#1)}}}
\makeatother

\makeatletter
\def\ps@pprintTitle{%
\let\@oddhead\@empty
\let\@evenhead\@empty
\let\@oddfoot\@empty
\let\@evenfoot\@oddfoot
}\makeatother

\begin{document}

\begin{frontmatter}

\title{A note on some constants related to the zeta--function \\ and their relationship with the Gregory coefficients}

\author{Iaroslav V.~Blagouchine\corref{cor1}} 
\ead{iaroslav.blagouchine@pdmi.ras.ru, iaroslav.blagouchine@univ-tln.fr}
\cortext[cor1]{Corresponding author.}
\address{Steklov Institute of Mathematics at St.-Petersburg, Russia \\
\& University of Toulon, France.} 

\author{Marc--Antoine Coppo\corref{cor2}}
\ead{coppo@unice.fr}
\address{Université Côte d'Azur, CNRS, LJAD (UMR 7351), France.}

\begin{abstract}
In this article, new series for the first and second Stieltjes constants (also known as generalized Euler's constant),
as well as for some closely related constants
are obtained. These series contain rational terms only and involve the so--called Gregory coefficients, 
which are also known as (reciprocal) logarithmic numbers, 
Cauchy numbers of the first kind and Bernoulli numbers of the second kind. In addition, 
two interesting series with rational terms for Euler's constant $\gamma$ and the constant $\ln2\pi$ are given, and yet another generalization of Euler's constant is proposed and various formulas for the calculation of these constants are obtained. 
Finally, we mention in the paper that almost all the constants 
considered in this work admit simple representations via the Ramanujan summation.
\end{abstract}

\end{frontmatter}

\section{Introduction and definitions}
The zeta-function 
\be\notag
\,\zeta(s) \equiv\sum\limits_{n=1}^\infty n^{-s}
\,= \prod\limits_{n=1}^\infty \!\big(1-p_n^{-s} \big)^{-1} 
\,\,,\qquad
\begin{array}{l}
\Re s>1\, \\[1mm]
p_n\in\mathbbm{P}\equiv\{2,3,5,7,11,\ldots\}
\end{array}
\ee
is of fundamental and long-standing importance in analytic number theory, modern analysis,
theory of $L$--functions, prime number theory and in a variety of other
fields. The $\zeta$--function is a meromorphic function on the entire complex
plane, except at point $s=1$ at which it has one simple pole with residue 1. 
The coefficients of the regular part of its Laurent series, denoted $\gamma_m$,
\be\label{984hd83ghsd}
\zeta(s)\,=\,\frac{1}{\,s-1\,} + \gamma + \sum_{m=1}^\infty \frac{(-1)^m (s-1)^m}{m!} \gamma_m\,, 
\qquad\qquad \qquad s\neq1.
\ee
where $\gamma$ is Euler's constant\footnote{We recall 
that $\,\gamma=\lim_{n\to\infty}(H_n - \ln n)=-\Gamma'(1)=0.5772156649\ldots\,$, where $H_n$ is the
harmonic number.},
and those of the Maclaurin series $\delta_m$
\be\label{i23udhjn2}
\zeta(s)\,=\,\frac{1}{\,s-1\,} + \frac{1}{2} + \sum_{m=1}^\infty \frac{(-1)^m s^m}{m!} \delta_m\,, 
\qquad\qquad \qquad s\neq1.
\ee
are of special interest and have been widely studied in literature,
see e.g.~\cite{gram_01}, \cite[vol.~I, letter 71 and following]{stieltjes_01}, \cite[p.~166 \emph{et seq.}]{finch_01},
\cite{jensen_02,jensen_03,franel_01,hardy_03,ramanujan_01,briggs_02,berndt_02,todd_01,israilov_01,zhang_01,coppo_01,iaroslav_09,lehmer_02,sitaramachandrarao_01,connon_08}. 
Coefficients $\gamma_m$ are 
usually called \emph{Stieltjes constants} or \emph{generalized Euler's constants} (both names being in use), while $\delta_m$ do not possess a 
special name.\footnote{It follows from \eqref{i23udhjn2} that $\,\delta_m=(-1)^m\big\{\zeta^{(m)}(0) + m!\big\}\,$} It may be shown with the aid of the Euler--MacLaurin summation
that $\gamma_m$ and $\delta_m$ may be also given by the following asymptotic representations
\begin{eqnarray} 
\gamma_m = \lim_{n\to\infty} \left\{
\sum_{k=1}^n \frac{\ln^m k}{k} -
\frac{\ln^{m+1} n}{m+1} \right\} , \quad m=1, 2, 3,\ldots\,,\quad\footnotemark
\end{eqnarray}
and\footnotetext{This representation is very old and was already known to Adolf Pilz, Stieltjes, Hermite and Jensen
\cite[p.~366]{iaroslav_09}.}
\begin{eqnarray}
\label{k98y9g87fcfcf} 
\delta_m \,= \lim_{n\to\infty} \!\left\{
\sum_{k=1}^n \ln^m \! k\, - 
n\,m!\!\sum_{k=0}^m (-1)^{m+k}\frac{\,\ln^k \! n\,}{k!}\,
- \frac{\,\ln^m \! n\,}{2} \right\}, \quad m=1, 2, 3,\ldots	\,,	\quad\footnotemark
\end{eqnarray}
\footnotetext{A slightly different 
expression for $\delta_m$ was given earlier by and by Lehmer \cite[Eq.~(5), p.~266]{lehmer_02}, Sitaramachandrarao 
\cite[Theorem 1]{sitaramachandrarao_01},
Finch \cite[p.~168 \emph{et seq.}]{finch_01} and Connon \cite[Eqs.~(2.15), (2.19)]{connon_08}.
The formula given by these writers differ from our \eqref{k98y9g87fcfcf} by the presence of the definite integral of $\ln^m \! x $
taken over $[1,n]$, which in fact may be reduced to a finite combination of logarithms and factorials.} 
These representations may be translated into these simple expressions
\begin{eqnarray}\notag
\gamma_m \,= \sum_{k\geqslant 1}^{\mathcal{R}} \frac{\ln^m k}{k} \,, \qquad
\delta_m \,= \sum_{k\geqslant 1}^{\mathcal{R}} \ln^m k \,,\qquad
m=1, 2, 3,\ldots
\end{eqnarray}
where $\sum\limits^{\mathcal{R}}$  stands for the sum of the series in the sense of the Ramanujan summation of divergent series \footnote{For more details on the \emph{Ramanujan summation},
see \cite[Ch.~6]{berndt_03}, \cite{candelpergher_02,candelpergher_01,candelpergher_03,candelpergher_04}.}.
Due to the reflection formula for the zeta-function
$\,\zeta(1-s)=2\zeta(s)\Gamma(s) (2\pi)^{-s}\cos\frac{1}{2}\pi s\,$, numbers $\delta_m$ and $\gamma_m$ 
are related to each other polynomially and also involve Euler's constant $\gamma$
and the values of the $\zeta$--function at naturals. For the first values of $m$, this gives
\be\label{976v076076}
\begin{array}{ll}
&\delta_1\,=\,\frac{1}{2}\ln2\pi -1 \,=\,-0.08106146679\ldots\\[2mm]
&\delta_2\,=\,\gamma_1 + \frac{1}{2}\gamma^2 - \frac{1}{2}\ln^2 2\pi -\frac{1}{24}\pi^2+2\,=\,-0.006356455908\ldots\quad\footnotemark   \\[2mm]
&\delta_3\,=\,	-\frac32\gamma_2 -  3\gamma_1\gamma  - \gamma^3 
- \big(3\gamma_1 +\frac32\gamma^2 -\frac18\pi^2 \big)\ln2\pi+ \zeta(3) +\frac12\ln^3 2\pi - 6 \,=\,+0.004711166862\ldots \\[2mm]
&\text{and conversely}		\\[2mm]
&\gamma_1\,=\,\delta_2 +2\delta_1^2+4\delta_1-\frac12\gamma+\frac{1}{24}\pi^2\,=\,-0.07281584548\ldots \\[2mm]
&\gamma_2\,=\, -\frac23\delta_3 - 2\delta_2 (\gamma+2) - 4\delta_1\delta_2 - \frac{16}{3}\delta_1^3
- 4\delta_1^2(\gamma+4) - 8\delta_1(\gamma+1) - \frac{1}{12}\gamma\pi^2
+ \frac13\gamma^3+\frac23\zeta(3)- \frac43 \\[2mm]
&\phantom{\gamma_1}\,=\,-0.009690363192\ldots \notag
\end{array}
\ee
\footnotetext{This expression for $\delta_2$ was found by Ramanujan,
see e.g.~\cite[(18.2)]{berndt_03}.}
Relationships between higher--order coefficients become very cumbersome, but may be found via a semi--recursive procedure described in 
\cite{apostol_02}.
Altough there exist numerous representations for $\gamma_m$ and $\delta_m$, 
no convergent series with rational terms only are known for them
(unlike for Euler's constant $\gamma$, see e.g.~\cite[Sect.~3]{iaroslav_09}, or for various expressions containing it
\cite[p.~413, Eqs.~(41), (45)--(47)]{iaroslav_08}).
Recently, divergent envelopping series for $\gamma_m$ containing rational terms only have been obtained in \cite[Eqs.~(46)--(47)]{iaroslav_09}.
In this paper, by continuing the same line of investigation, we derive convergent series representations with rational coefficients 
for  $\gamma_1, \delta_1, \gamma_2$ and $\delta_2$, and also find two new series of the same type for Euler's constant $\gamma$
and $\ln2\pi$ respectively.
These series are not simple and involve a product of \emph{Gregory coefficients} $G_n$, which are also known 
as \emph{(reciprocal) logarithmic numbers}, \emph{Bernoulli numbers of
the second kind} $b_n$, and normalized \emph{Cauchy numbers of the first kind} $C_{1,n}$.
Similar expressions for higher--order constants $\gamma_m$ and $\delta_m$ may be obtained by the same procedure, 
using the harmonic product of sequences introduced in \cite{candelpergher_05}, but are quite cumbersome.
Since the Stieltjes constants $\gamma_m$ generalize Euler's constant $\gamma$ and since our series
contain the product of $G_n$, these new series
may also be seen as the generalization of the famous Fontana--Mascheroni series
\begin{equation}\label{923jsd3c}
\begin{array}{ll}
\gamma \, & \displaystyle= \sum_{n=1}^{\infty}\!\frac{\, \big|G_n \big|\,}{n} \,=\,\frac{1}{2}+\frac{1}{24}+\frac
{1}{72}+\frac{19}{2880}+\frac{3}{800}
+\frac{863}{362\,880}+\frac{275}{169\,344} + \ldots
\end{array}
\end{equation}
which is the first known series representation for  Euler's constant having rational terms only, see \cite[pp.~406, 413, 429]{iaroslav_08}, \cite[p.~379]{iaroslav_09}. 
In Appendix A, we introduce yet another set of constants $\kappa_m=\sum\limits_{n\geqslant1} |G_n|\, n^{-m-1}$, which 
also generalize Euler's constant $\gamma$. 
These numbers, similarly to $\gamma_m$,
coincide with Euler's constant at $m=0$ and have various interesting series and integral representations, none of them
being reducible to the ``classical'' mathematical constants.

\section{Series expansions}
\subsection{Preliminaries}
Since the results, that we come to present here, are essentially based on the Gregory coefficients and Stirling numbers,
it may be useful to briefly recall their definition and properties. Gregory numbers, denoted below $G_n$,
are rational alternating $G_1=+\nicefrac{1}{2}\,$,
\mbox{$G_2=\nicefrac{-1}{12}\,$}, $G_3=\nicefrac{+1}{24}\,$, $G_4=\nicefrac{-19}{720}\,$,
$G_5=\nicefrac{+3}{160}\,$, $G_6=\nicefrac{-863}{60\,480}\,$,\ldots \,,
decreasing in absolute value, and are also closely related to the theory of finite differences; 
they behave as $\,\big(\,n\ln^2 n\big)^{-1}\,$ at $n\to\infty$ and may be bounded from below and above accordingly
to \cite[Eqs.~(55)--(56)]{iaroslav_08}.
They may be defined either via their generating function
\begin{eqnarray}
\label{eq32}
\frac{z}{\ln(1+z)} = 1+\sum
_{n=1}^\infty G_n \, z^n ,\qquad|z|<1\,,
\end{eqnarray}
or recursively
\be\label{jfhek4s}
G_n\,=\,\frac{(-1)^{n+1}}{\,n+1\,}\,+\sum_{l=1}^{n-1}\! \frac{(-1)^{n-l}G_l}{\,n+1-l\,}\,,
\qquad G_1=\frac{1}{2}\,,\quad n=2,3,4,\ldots
\ee
or explicitly\footnote{For more information
about $G_n$, see \cite[pp.~410--415]{iaroslav_08}, \cite[p.~379]{iaroslav_09}, \cite{iaroslav_11},
and the literature given therein (nearly 50 references).}
\begin{eqnarray}
G_n\,=
\,\frac{1}{n!}\! \int\limits_0^1\! x\,(x-1)\,(x-2)\cdots(x-n+1)\, dx\,
,\qquad n=1,2,3,\ldots
\end{eqnarray}
Throughout the paper, 
we also make use of the Stirling numbers of the first kind, which we denote below by $S_1(n,l)$. 
Since there are different definitions and notations for them, we specify that in our definition they are simply 
the coefficients in the expansion of falling factorial
\be\label{x2l3dkkk03d}
x\,(x-1)\,(x-2)\cdots(x-n+1)\,=\,\sum_{l=1}^n S_1(n,l)\cdot x^l \,,\qquad n=1, 2, 3,\ldots
\ee
and may equally be defined via the generating function
\be\label{ld2jr3mnfdmd}
\frac{\ln^l(1+z)}{l!}\,=\sum_{n=l}^\infty\!\frac{S_1(n,l)}{n!}z^n \,=
\sum_{n=0}^\infty\!\frac{S_1(n,l)}{n!}z^n \,, \qquad  |z|<1\,,\quad l=0, 1, 2, \ldots \\[6mm]
\ee 
It is important to note that $\sgn\big[ S_1(n,l)\big] =(-1)^{n\pm l}$.
We also recall that the Stirling numbers of the first kind and the Gregory coefficients are linked by the following relation\footnote{More
information and references (more than 60) on the Stirling numbers of the first kind
may be found in \cite[Sect.~2.1]{iaroslav_08} and \cite[Sect.~1.2]{iaroslav_09}. We also note that our definitions for the Stirling numbers agree with 
those adopted by \textsc{Maple} or \textsc{Mathematica}: our $S_1(n,l)$ equals to \texttt{Stirling1(n,l)} from the former
and to \texttt{StirlingS1[n,l]} from  the latter.}
\begin{eqnarray}\label{ldhd9ehn}
G_n\,=
\,\frac{1}{\,n!\,}\!\sum_{l=1}^n \frac{\,S_1(n,l)\,}{l+1}
,\qquad n=1,2,3,\ldots
\end{eqnarray}

\subsection{Some auxiliary lemmas}
Before we proceed with the series expansions for $\delta_m$ and $\gamma_m$,
we need to prove several useful lemmas. 

\begin{lemma}
For each natural number $k$, let 
\[
\sigma_k\equiv\sum_{n=1}^{\infty}\!\frac{\big|\,G_n\big|}{\,n+k\,}\,,
\]
the following equality holds
\be
\,
\sigma_k =\,\frac{1}{k}+\sum_{m=1}^{k} (-1)^m \binom{k}{m}\ln(m+1)\,,\qquad k=1, 2, 3,\ldots
\label{s2kj0sdw} 
\ee
\end{lemma}

\begin{proof}
By using \eqref{ldhd9ehn} and by making use of the generating equation
for the Stirling numbers of the first kind \eqref{ld2jr3mnfdmd}, we obtain
\be\label{jkd293ndnw}
\begin{array}{l}
\displaystyle
\sum_{n=1}^{\infty}\!\frac{\big|\,G_n\big|}{\,n+k\,}\,
=\sum_{n=1}^{\infty} \frac{1}{\,n!\,} \sum_{l=1}^{n}\frac {(-1)^{l+1}\big|S_1(n,l)\big|}{l+1}\cdot
\!\!\underbrace{\,\int\limits_0^1 \! x^{n+k-1} dx}_{1/(n+k)}\,
=\,-\!\sum_{l=1}^{\infty} \frac{1}{\,(l+1)!\,}\! \int\limits_0^1 \! x^{k-1}\!
\ln^{l}(1-x)\,  dx \\[8mm]
\displaystyle\qquad
=\sum_{l=1}^{\infty} \frac{(-1)^{l+1}}{\,(l+1)!\,}\int\limits_0^\infty\! 
\big(1-e^{-t}\big)^{k-1} t^l e^{-t}\,dt\,=
\sum_{l=1}^{\infty} \frac{(-1)^{l+1}}{\,(l+1)!\,}\sum_{m=0}^{k-1}
(-1)^m\binom{k-1}{m}\underbrace{\int\limits_0^\infty\!e^{-t(m+1)}t^l\,dt}_{l! \,(m+1)^{-l-1}}\\[8mm]
\displaystyle\qquad
=
\sum_{l=1}^{\infty} \frac{(-1)^{l+1}}{\,l+1\,}\sum_{m=0}^{k-1}
(-1)^m\binom{k-1}{m}\frac{1}{(m+1)^{l+1}} \,=
\sum_{m=0}^{k-1} (-1)^m \binom{k-1}{m}\left\{\frac{1}{m+1}-\ln\frac{m+2}{m+1} \right\}
\end{array}
\ee
where at the last stage we made a change of variable $\,x=1-e^{-t}\,$
and used the well-known formula for the $\Gamma$--function. But since
\be\notag
\sum\limits_{m=0}^{k-1} \frac{(-1)^m}{m+1} \binom{k-1}{m}\,=\,\frac{1}{\,k\,}\,,
\qquad\text{and}
\qquad\binom{k-1}{m}+\binom{k-1}{m-1}\,=\,\binom{k}{m}\,,
\ee
the last finite sum in \eqref{jkd293ndnw} reduces to \eqref{s2kj0sdw}.\footnote{This result 
appeared without proof in \cite[p.~413]{iaroslav_08}. 
For a slightly more general result, see \cite[Proposition 1]{coppo 02}.}
\end{proof}

\begin{remark}
One may show \footnote{See \cite{candelpergher_04} Eq.~(4.31).}that $\sigma_k$ may also be written in terms of the Ramanujan summation:
\be
\sigma_{k} = \sum_{n\geqslant 1}^{\mathcal{R}} \frac{\Gamma(k+1) \Gamma(n)}{\Gamma(n+k+1)}
= \sum_{n\geqslant 1}^{\mathcal{R}} \textit{B}(k+1,n).
\ee
where $\textit{B}$ stands for the Euler beta-function.
\end{remark}

\begin{lemma} Let $a = (a(1), a(2), \dots, a(n), \dots )$ be a sequence of complex numbers. The following identity is true for all nonnegative integers $n$:
\begin{equation}\label{F2}
\sum_{l=0}^{n} (-1)^l \binom{n}{l} \frac{a(l+1)}{l+1} = \frac1{n+1} \sum_{k=0}^{n} \sum_{l=0}^{k}(-1)^l \binom{k}{l}a(l+1)
 \end{equation}
In particular, if $a = \ln^m$ for any natural $m$, then this identity reduces to 
\begin{equation}\label{F3}
 \sum_{l=0}^{n} (-1)^l \binom{n}{l} \frac{\ln^m(l+1)}{l+1} = \frac1{n+1} \sum_{k=1}^{n} \sum_{l=1}^{k}(-1)^l \binom{k}{l}\ln^m(l+1)
\end{equation}
\end{lemma}
\begin{proof}
Formula (\ref{F2}) is an explicit translation of \cite[Proposition 7]{candelpergher_05}. 
\end{proof}

\begin{lemma}
For all natural $m$ 
\begin{align}\label{nhfo2ihf}
\gamma_m & =\sum_{n=1}^{\infty}\frac{|G_{n+1}|}{n+1} \sum_{k=1}^n 
\sum_{l=1}^{k}(-1)^l \binom{k}{l}\ln^m(l+1)\\[2mm]
\delta_m & = \sum_{n=1}^{\infty}|G_{n+1}|\sum_{l=1}^{n} (-1)^l \binom{n}{l} \ln^m(l+1)
\end{align}
\end{lemma}
\begin{proof}
Using this representation for the $\zeta$--function 
\begin{eqnarray}\nonumber
\zeta(s)\,=\,\frac{1}{\,s-1\,}+\sum_{n=0}^\infty\big| G_{n+1}\big
|\sum_{k=0}^{n} (-1)^k \binom{n}{k}(k+1)^{-s}\,,\qquad s\neq1\,,
\end{eqnarray}
see e.g.~\cite[pp.~382--383]{iaroslav_09}, \cite{iaroslav_10}, we first have
\be\notag
\gamma_m =\sum_{n=0}^{\infty} \big|G_{n+1}\big|\sum_{l=0}^{n} (-1)^{l} \binom{n}{l}
\frac{\ln^m(l+1)}{l+1}
 \ee
and 
\be\notag
\delta_m =\sum_{n=0}^{\infty}\big|G_{n+1}\big|\sum_{l=0}^{n} (-1)^{l}\binom{n}{l} \ln^m(l+1)\,.
\ee
Then formula \eqref{nhfo2ihf} follows from property \eqref{F3}. 

\end{proof}

\subsection{Series with rational terms for the first Stieltjes constant $\gamma_1$ and for the coefficient $\delta_1$}
\begin{theorem}\label{h9238dnd}
The first Stieltjes constant $\gamma_1$ may be given by the following series
\begin{eqnarray}
\gamma_1 && 
=\,\frac{3}{2}-\frac{\pi^2}{6}\,+\,
\sum_{n=2}^\infty \left[\frac{\big|G_n\big|}{n^2}\,+\,\sum_{k=1}^{n-1}\frac{\big|G_k\, G_{n+1-k}\big|\cdot\big(H_n - H_k \big) }{n+1-k}
\right] \notag\\[3mm]
&& 
= \frac{3}{2} -\frac{\pi^2}{6} +  \frac{1}{32} +  \frac{5}{432} +  \frac{1313}{207\,360} 
+  \frac{42\,169}{10\,368\,000} + \frac{137\,969}{48\,384\,000} + \frac{1\,128\,119}{5\,334\,336\,00}+  \ldots \label{kjhx2e9nbd}
\end{eqnarray}
containing $\pi^2$ and positive rational coefficients only. Using Euler's formula $\,\pi^2=6\sum n^{-2}\,$, 
the latter may be reduced to a series with rational terms only.
\end{theorem}

\begin{proof}
By \eqref{nhfo2ihf} with $m=1$, one has 
\be\notag
\gamma_1 =\sum_{n=1}^{\infty}\frac{\big|G_{n+1}\big|}{n+1} \sum_{k=1}^n \sum_{m=1}^{k}(-1)^m \binom{k}{m}\ln(m+1)
 \ee
and by \eqref{s2kj0sdw}, 
\be\notag
\frac{1}{n+1} \sum_{k=1}^n \sum_{m=1}^{k}(-1)^m \binom{k}{m}\ln(m+1) =\frac1{n+1} \sum_{k=1}^{n} \sigma_k  -  \frac{H_{n+1}}{n+1} + \frac1{(n+1)^2}\,. 
\ee
Thus  
\begin{align*}
\gamma_1 & =  \sum_{n=1}^{\infty}\frac{\big|G_{n+1}\big|}{n+1} \sum_{k=1}^n \sigma_k - 
\sum_{n=0}^{\infty}\frac{\big|G_{n+1}\big| \, H_{n+1}}{n+1} + \sum_{n=0}^{\infty}\frac{\big|G_{n+1}\big|}{(n+1)^2}\\[6mm]
& =  \sum_{n=1}^{\infty} \sum_{m=1}^{\infty}\frac{\big|G_{n+1}\, G_m\big|(H_{m+n} -H_m) }{n+1} -\zeta(2) + 1
+ \sum_{n=1}^{\infty}\frac{\big|G_n\big|}{n^2}
\end{align*}
since 
\be\notag
\sum_{n=1}^{\infty}\frac{\big|G_{n}\big|\cdot H_{n}}{n} = \zeta(2) -1
 \ee
see e.g.~\cite[p.~2952, Eq.~(1.3)]{young_01}, \cite[p.~20,Eq.~(3.6)]{coffey_07}, 
\cite[p.~307, Eq.~for $F_0(2)$]{candelpergher_01}, \cite[p.~413, Eq.~(44)]{iaroslav_08}.
Rearranging the double absolutely convergent series as follows
\be\notag
\sum_{n=1}^{\infty} \sum_{m=1}^{\infty}\frac{\big|G_{n+1}\, G_m\big|\cdot\big(H_{m+n} -H_m\big)}{n+1}
=  \sum_{n=2}^{\infty} \sum_{k=1}^{n-1}\frac{\big|G_k\, G_{n+1-k}\big|\cdot\big(H_n - H_k \big) }{n+1-k}
 \ee
we finally arrive at \eqref{kjhx2e9nbd}.
\end{proof}

\begin{remark}\label{jf3904jdf}
It seems that the sum $\,\kappa_1\equiv\sum |G_n|\, n^{-2}\,=0.5290529699\ldots\,$ cannot
be reduced to the ``standard'' mathematical constants.\footnote{For more digits, see OEIS A270859.} However, it admits several interesting representations, which we give in Appendix A.
\end{remark}

\begin{theorem}\label{hiyg8tvr}
The first MacLaurin coefficient $\,\delta_1=\,\frac{1}{2}\ln2\pi -1\,$ admits a series representation similar to that for $\gamma_1$, namely
\be\label{78y897ycxq2}
\delta_1 = \sum_{n=1}^{\infty}\frac{1}{\,n\,}\sum_{k=1}^{n} \big|G_k\,G_{n+1-k}\big|+ \frac{1 -\ln2\pi}{2}
\ee
\end{theorem}

\begin{proof}
Proceeding analogously to the previous case and recalling that 
\be\notag
\sum_{n=2}^{\infty}\frac{|G_{n}|}{n-1}\,=\,  -\frac{\gamma +1 -\ln2\pi}{2}
\ee
see e.g.~\cite[p.~413, Eq.~(41)]{iaroslav_08}, \cite[Corollary 9]{young 02}, we have
\begin{align}
\delta_1 
& = \sum_{n=0}^{\infty}\big|G_{n+1}\big| \sum_{l=0}^n (-1)^l \binom{n}{l} \ln(l+1)\,\notag\\[3mm]
& =  \sum_{n=1}^{\infty}\big|G_{n+1}\big|\left(\sigma_n -\frac1n\right)
=  \sum_{n=1}^{\infty}\big|G_{n+1}\big| \, \sigma_n -  \sum_{n=1}^{\infty}\frac{\big|G_{n+1}\big|}{n} \,\notag\\[3mm]
& =  \sum_{n=1}^{\infty}\sum_{k=1}^{\infty} \frac{\big|G_{n+1}\,G_k\big|}{n+k} + \frac{\gamma +1 -\ln2\pi}{2}\,=
\sum_{n=1}^{\infty}\sum_{k=1}^{\infty} \frac{\big|G_n\,G_k\big|}{n+k-1} + \frac{1 -\ln2\pi}{2}\,\label{0984rjnf}\\[3mm]
& =  \sum_{n=1}^{\infty}\frac{1}{\,n\,}\sum_{k=1}^{n} \big|G_k\,G_{n+1-k}\big|+ \frac{1 -\ln2\pi}{2}\notag
\end{align}
where in \eqref{0984rjnf} we could eliminate $\,\gamma\,$ thanks to the fact that 
$G_1=\nicefrac{1}{2}$ and that the sum of $\,|G_n|/n\,$ over all natural $n$ equals precisely Euler's constant, see \eqref{923jsd3c}.
\end{proof}

\begin{corollary}
The constant $\ln2\pi$ has the following beautiful series representation with rational terms only and
containing a product of Gregory coefficients
\be
\ln2\pi \,=\,\frac{3}{2} + \!\sum_{n=1}^{\infty}\frac{1}{\, n\,}\sum_{k=1}^n \big|G_k\, G_{n+1-k}\big| \,=
\,\frac{3}{2} + \frac{1}{4}+\frac{1}{24}+\frac{7}{432}+\frac{1}{120}+\frac{43}{8640}+\frac{79}{24\,192}+ \ldots
\ee
which directly follows from \eqref{78y897ycxq2}. From the latter, one can also readily derive a series
with rational coefficients only for $\ln\pi$ (for instance, with the aids of
the Mercator series).
\end{corollary}

\begin{corollary}
Euler's constant $\gamma$ admits the following series representation with rational terms 
\begin{eqnarray}
\gamma && 
=\,2\ln2\pi - 3 - 2 \!\sum_{n=1}^{\infty}\frac{1}{n+1}\sum_{k=1}^n \big|G_k\, G_{n+2-k}\big| \,=
\,2\ln2\pi - 3 - \frac{1}{24}-\frac{1}{54}-\frac{29}{2880}- \notag\\[2mm]
 &&\quad
-\frac{67}{10\,800}-\frac{1507}{362\,880}-\frac{3121}{1\,058\,400}-\frac{12\,703}{15\,806\,080}
 -\frac{164\,551}{97\,977\,600}-\ldots\label{897v965c}
\end{eqnarray}
which seems to be undiscovered yet. 
This curious series straightforwardly
follows from \eqref{0984rjnf}, from the transformation
\be\notag
\sum_{k=1}^{\infty}\sum_{n=1}^{\infty} \frac{\big|G_{k+1}\,G_n\big|}{n+k} \,=
\sum_{n=1}^{\infty}\frac{1}{n+1}\sum_{k=1}^n \big|G_k\, G_{n+2-k}\big| 
\ee
and from Eq.~\eqref{976v076076}.
\end{corollary}

\subsection{Generalizations to the second--order coefficients $\delta_2$ and $\gamma_2$ via an application of the harmonic product}
We recall the main properties of the harmonic product of sequences which are stated and proved in \cite{candelpergher_05}).
If $a = (a(1), a(2), \dots)$ and $b=  (b(1), b(2), \dots)$ are two sequences in $\mathbb{C}^{\mathbb{N^*}}$, then
the harmonic product $a \Join b$ admits the explicit expression:
\begin{equation}\label{F7}
(a\Join b)(m+1)=\sum_{0\leqslant l \leqslant k \leqslant m}(-1)^{k-l}\binom{m}{k}\binom{k}{l}\,a(k+1)b(m+1-l)\,, 
\qquad m=0, 1, 2,\ldots
\end{equation}
For small values of $m$, this gives: 
\begin{align*}
(a\Join b)(1)& =a(1)b(1)\,, \\
(a\Join b)(2)& =a(2)b(1)+a(1)b(2)-a(2)b(2)\,, \\
(a\Join b)(3)& =a(3)b(1)+a(1)b(3)+2a(2)b(2)-2a(3)b(2)-2a(2)b(3)+a(3)b(3)\,\\
\text{ etc.} &
\end{align*}
The harmonic product $\Join$ is associative and commutative. 

Let $D$ be the operator defined by 
\be\notag
D(a)(m+1) = \sum_{j=0}^{m} (-1)^j \binom{m}{j} a(j+1)\,,\qquad m=0, 1, 2,\ldots
\ee
then $D = D^{-1}$ and the harmonic product satisfies the following property:
\begin{equation}\label{F8}
 D(ab) = D(a) \Join D(b)   
\end{equation}
In particular, if $a(m) = \ln m$, then $D(a)(1) = \ln1 = 0$, and by \eqref{s2kj0sdw}, 
\begin{equation}\label{F9}
D(a)(m+1) = \sum_{j=1}^{m} (-1)^j \binom{m}{j}\ln(j+1) = \sigma_m - \frac1m\,, \qquad m=1, 2, 3,\ldots
 \end{equation}
Therefore, if $a = \ln$ then, by (\ref{F7}), (\ref{F8}), and (\ref{F9}), the following identity holds
\begin{equation}\label{F10}
D(\ln^2)(m+1) = \sum_{\substack{0\leqslant l \leqslant k \leqslant m\\
k\neq 0\\ l\neq m}}(-1)^{k-l}\binom{m}{k}\binom{k}{l}\,\left(\sigma_k - \frac1k\right)
\!\!\left(\sigma_{m-l} - \frac1{m-l}\right)
\end{equation}
From this identity results  the following theorem: 
\begin{theorem}
The second coefficients $\gamma_2$ and $\delta_2$ may be given by the following series
\begin{equation*}
\gamma_2 = \sum_{n=1}^{\infty}\frac{|G_{n+1}|}{n+1} \sum_{\substack{0\leqslant l \leqslant k \leqslant m \leqslant n\\ k\neq 0\\ l \neq m}}
(-1)^{k-l}\binom{m}{k}\binom{k}{l}\left(\sigma_k -\frac1k\right)\!\! \left(\sigma_{m-l} -\frac1{m-l}\right) \, 
\end{equation*}
and 
\begin{align*}
\delta_2 
= \sum_{m=1}^{\infty}|G_{m+1}| \sum_{\substack{0\leqslant l \leqslant k \leqslant m \\ k\neq 0\\ l\neq m}}(-1)^{k-l}
\binom{m}{k}\binom{k}{l}\left(\sigma_k -\frac1k\right)\!\! \left(\sigma_{m-l} -\frac1{m-l}\right) \,, 
\end{align*}
respectively.
\end{theorem}

\begin{proof} Applying  \eqref{nhfo2ihf} with $m=2$, and using equation (\ref{F10}), we can write the following equalities:  
 
\begin{align*}
\gamma_2 & = \sum_{n=1}^{\infty}\frac{|G_{n+1}|} {n+1} \sum_{m=1}^{n} \sum_{j=1}^{m}(-1)^j \binom{m}{j}\ln^2(j+1)\\[2mm]
& = \sum_{n=1}^{\infty}\frac{|G_{n+1}|} {n+1} \sum_{m=1}^{n} D(\ln^2)(m+1) \\[2mm]
& = \sum_{n=1}^{\infty}\frac{|G_{n+1}|} {n+1} \sum_{m=1}^{n} \, \sum_{\substack{0\leqslant l \leqslant k \leqslant m\\[2mm]
k\neq 0\\ l\neq m}}(-1)^{k-l}\binom{m}{k}\binom{k}{l}\,\left(\sigma_k - \frac1k\right)
\!\!\left(\sigma_{m-l} - \frac1{m-l}\right)\,,
\end{align*}
and for $\delta_2$,
\begin{align*}
\delta_2 & =  \sum_{n=0}^{\infty}|G_{n+1}| \sum_{j=0}^{n} (-1)^j \binom{n}{j}\ln^2(j+1)\\[1mm]
& = \sum_{n=1}^{\infty}|G_{n+1}| D(\ln^2)(n+1)\\[1mm]
& = \sum_{m=1}^{\infty}|G_{m+1}| \sum_{\substack{0\leqslant l \leqslant k \leqslant m \\ k\neq 0\\ l\neq m}}
(-1)^{k-l}\binom{m}{k}\binom{k}{l}\left(\sigma_k -\frac1k\right) \!\!\left(\sigma_{m-l} -\frac1{m-l}\right) \,.
\end{align*}
\end{proof}

By following the same method, one may also obtain 
expressions for higher--order constants $\gamma_m$ and $\delta_m$. 
However, the resulting expressions are more theoretical than practical.

\section*{Acknowledgments} 
The authors are grateful to Vladimir V.~Reshetnikov for his kind help and useful remarks.

\section*{Appendice A. \, Yet another generalization of Euler's constant}
The numbers $\kappa_p\equiv\sum|G_n|\,n^{-p-1}$, where the summation extends over $n=[1,\infty)$, 
may also be regarded as one of the possible generalizations of Euler's constant
(since $\kappa_0=\gamma_0=\gamma$ and $\kappa_{-1}=\gamma_{-1}=1$).\footnote{Numbers $\kappa_0$ and 
$\kappa_{-1}$ are found for the values
to which Fontana--Mascheroni and Fontana series converge respectively \cite[pp.~406, 410]{iaroslav_08}.}\up{,}\footnote{Other
possible generalizations of Euler's constant were proposed by Briggs, Lehmer, Dilcher and some other authors
\cite{briggs_01,lehmer_01,tasaka_01,pilehrood_01,xia_01,dilcher_01}.} 
These constants, which do not seem to be reducible to the ``classical mathematical constants'',
admit several interesting representations as stated in the following proposition.

\begin{proposition}
Generalized Euler's constants $\kappa_{p}\equiv\sum|G_n|\,n^{-p-1}$, where the summation extends over $n=[1,\infty)$, admit the following representations:
\begin{eqnarray}
\kappa_{p}&&\displaystyle =
\,\frac{(-1)^{p}}{\Gamma(p+1)}\int\limits_0^1 \! \left\{ \frac{1}{\ln(1-x)}+ \frac{1}{x} \right\}\ln^{p}\! x\; dx\,, 
\quad \Re p > -1\,.\label{9873rhw}\\[3mm]
&&\displaystyle = \underbrace{\int\limits_0^1\!\cdots\!\int\limits_0^1\! }_{p\text{-fold}} 
\!\left\{\li\!\left(\!1-\prod_{k=1}^{p}x_k \!\right)+\gamma
+ \sum_{k=1}^{p}\ln x_k \right\}\, \frac{dx_1 \cdots dx_{p}}{x_1 \cdots x_{p}}\,, 
\quad p = 1, 2, 3,\ldots\label{jhgwuyg3}\\[3mm]
&&\displaystyle = \sum_{k=2}^{\infty} \frac{(-1)^{k}}{k} \sum_{n=p+1}^{\infty} \frac{\big|S_1(n,p+1)\big|}{n!\,n^{k-1}}\,, 
\quad p = 0, 1,2,\ldots\label{g8t6v877}\\[3mm]
&&\displaystyle = \sum_{k=2}^{\infty} \frac{(-1)^{k}}{k}\!\!\!
\sum_{n=p}^{\infty}\, \frac{P_{p}(H_{n}^{(1)}, -H_{n}^{(2)},\dots, (-1)^{p-1}H_{n}^{(p)}) }{(n+1)^{k}}\,, 
\quad p = 0, 1, 2,\ldots\label{g8t6v877b}\\[3mm]
&&\displaystyle =
\sum_{n\geqslant 1}^{\mathcal{R}} \frac1n \sum_{n\geqslant n_{1}\geqslant \ldots \geqslant n_{p}\geqslant 1}\frac{1}{n_{1}\ldots n_{p}} 
= \sum_{n\geqslant 1}^{\mathcal{R}} \frac{P_{p}\big(H_{n}^{(1)}, H_{n}^{(2)},\dots,H_{n}^{(p)}\big) }{n}\,,\quad p= 1, 2, 3, \ldots\qquad
\label{kq9jdn2dd}
\end{eqnarray}
where $\li$ is the integral logarithm function,  $H^{(m)}_n\equiv\sum\limits_{k=1}^n k^{-m}$ stands for the generalized harmonic number and  
$P_m$ denotes the modified Bell polynomials
\begin{align*}
& P_0 = 1,\quad  P_1(x_1) =  x_1, \quad  P_2(x_1, x_2) = \tfrac{1}{2}\left(x_1^2 + x_2\right),\\[1mm]
& P_3 (x_1, x_2, x_3) = \tfrac{1}{6}\left(x_1^3 + 3 x_1 x_2 + 2x_3\right), \quad \ldots \quad\footnotemark
\end{align*}
\footnotetext{More generally,
these polynomials are defined by the generating function: $\,\exp\left(\sum\limits_{n=1}^{\infty}x_n \frac{t^n}{n}\right) = 
\sum\limits_{m=0}^{\infty}P_m(x_1,\cdots,x_m)\, t^m \,$.}
In particular, for the series $\kappa_1$ which we encountered in Theorem \ref{h9238dnd} and Remark \ref{jf3904jdf}, this gives
\begin{eqnarray}
\displaystyle\kappa_1\,=\sum_{n=1}^{\infty}\frac{\big|G_n\big|}{n^2} \!\!
&&\displaystyle =\, -\int\limits_0^1 \! \left\{ \frac{1}{\ln(1-x)}+ \frac{1}{x} \right\}\ln x\; dx \\[3mm]
&&\displaystyle=\, \int\limits_0^1 \! \frac{\,-\li(1-x) + \gamma + \ln x\,}{x}\, dx \,
= \int\limits_0^\infty \!\! \Big\{\! -\li\big(1-e^{-x}\big) + \gamma - x \Big\}\, dx\\[3mm]
&&\displaystyle= 
\sum_{k=2}^{\infty} \frac{(-1)^{k}}{k}\sum_{n=2}^{\infty}\frac{H_{n-1}}{n^{k}}\,=
\sum_{n\geqslant 1}^{\mathcal{R}} \frac{H_n}{n} \,. \label{u2394udcdn} 
\end{eqnarray}
Moreover, we also have
\begin{eqnarray}
\displaystyle\kappa_1\!\!
&&\displaystyle =\gamma_1 + \frac{\gamma^2}{2} - \frac{\pi^2}{12}  + \int\limits_0^1 \! \frac{\Psi(x+1) + \gamma}{x}\,dx \label{h293hsmw87b}\\
&&\displaystyle =  \frac{\gamma^2}{2} + \frac{\pi^2}{12} -\frac12  + \frac12\!\int\limits_0^1 \!\Psi^2(x+1) \,dx \label{h293hsmw87}
\end{eqnarray}
where $\Psi$ denotes the digamma function (logarithmic derivative of the $\Gamma$--function).
\end{proposition}

\subsection*{Proof of formula \eqref{9873rhw}}
We first write the generating equation
for Gregory's coefficients, Eq.~\eqref{eq32}, in the following form
\be\label{8923hc2e}
\frac{1}{\ln (1-x)}+ \frac1x  \, =\sum_{n=1}^\infty |G_{n}|\, x^{n-1}\,, \qquad |x|<1\,.
\ee
Multiplying both sides by $\ln^{p} x$, integrating over the unit interval and changing the order of summation and integration\footnote{The series
being uniformly convergent.} yields:
\be\label{0934urnde}
\int\limits_0^1 \! \left\{ \frac{1}{\ln(1-x)}+ \frac{1}{x} \right\}\ln^{p}\! x\; dx\, =\,
\sum_{n=1}^\infty |G_{n}| \int\limits_0^1 \! x^{n-1} \ln^{p} \! x \; dx\,,\qquad \Re p>-1\,.
\ee
The last integral may be evaluated as follows. Considering Legendre's integral 
$\,\Gamma(p+1)\,=\int t^{p} e^{-t} dt\,$ taken over $[0,\infty)$ and
making a change of variable $\,t=-(s+1)\ln x\,$, we have
\be
\,\int\limits_0^1 \! x^s \ln^p \! x \; dx \,=\, (-1)^p\frac{\Gamma(p+1)}{\, (s+1)^{p+1}}\,,
\qquad 
\begin{array}{l}
\Re s>-1 \\[1mm]
\Re p>-1
\end{array}
\,.
\ee
Inserting this formula into \eqref{0934urnde} and setting $n-1$ instead of $s$, yields \eqref{9873rhw}.

\subsection*{Proof of formula \eqref{jhgwuyg3}}
Putting in \eqref{8923hc2e} $\,x=x_1x_2\cdots x_{p+1}\,$
and integrating over the volume $[0,1]^{p+1}$, where $p\in\mathbbm{N}$, on the one hand, we have
\be\label{ij023dj2ne}
\underbrace{\int\limits_0^1\!\cdots\!\int\limits_0^1\! }_{(p+1)\text{-fold}} 
\! \sum_{n=1}^\infty |G_n| \big(x_1x_2\cdots x_{p+1}\big)^{n-1}
dx_1 \cdots dx_{p+1} \,=
\sum_{n=1}^{\infty}\frac{|G_n|}{n^{p+1}} 
\ee
On the other hand
\be\notag
\int\limits_0^1 \! \left\{ \frac{1}{\ln(1-xy)}+ \frac{1}{xy} \right\}\, dx\, =\,-
\frac{\,\li(1-y) - \gamma - \ln y\,}{y}
\ee
Taking instead of $y$ the product $\,x_1x_2\cdots x_{p}\,$ and setting $\,x=x_{p+1}\,$, and then integrating $p$ times
over the unit hypercube and equating the result with \eqref{ij023dj2ne} yields \eqref{jhgwuyg3}.

\subsection*{Proof of formula \eqref{h293hsmw87b}--\eqref{h293hsmw87}}
Using Eqs.~(3.21)--(3.23) of \cite{candelpergher_04} we obtain \eqref{h293hsmw87b}--\eqref{h293hsmw87}.

\subsection*{Proof of formulas \eqref{g8t6v877}--\eqref{g8t6v877b}}
Writing in the generating equation \eqref{ld2jr3mnfdmd} $x$ instead of $z$, multiplying it by $\ln^m x/x$
and integrating over the unit interval, we obtain the following relation\footnote{See also  \cite[Theorem 2.7]{xu_01}.}
\be\notag
\Omega(k,m)  = (-1)^{m+k} m! \, k!  \sum_{n=k}^{\infty} \frac{\big|S_1(n,k)\big|}{n!\,n^{m+1}}
\ee
where
\be\notag
\Omega(k,m)  \, \equiv  \int\limits_0^1 \! \frac{\ln^k(1-x)\,\ln^m x}{x} \, dx\,,\qquad 
\begin{array}l{}
k\in\mathbbm{N}\\
m\in\mathbbm{N}
\end{array}
\ee
By integration by parts, it may be readily shown that 
\be\notag
\Omega(k,m) \,=\,\frac{k}{m+1}\Omega(m+1,k-1) 
\ee
and thus, we deduce the duality formula:
\be\notag
\sum_{n=k}^{\infty} \frac{\big|S_1(n,k)\big|}{n!\,n^{m+1}} = \sum_{n=m+1}^{\infty} \!\!\! \frac{\big|S_1(n,m+1)\big|}{n!\,n^{k}}\,. 
\ee
Now, writing
\be\notag
 x + \ln(1-x) = - \sum_{k=1}^{\infty}\frac{\ln^{k+1}(1-x)}{(k+1)!}
\ee
we obtain  
\begin{align*}
\int\limits_0^1 \! \left\{ \frac{1}{\ln(1-x)}+ \frac{1}{x} \right\}\ln^{m}\! x\; dx & 
= - \int\limits_0^1 \sum_{k=1}^{\infty}\frac{\ln^{k+1}(1-x)}{(k+1)!}\cdot\frac{\ln^m x}{\ln(1-x)}\cdot \frac{dx}{x}\\
& = - \sum_{k=1}^{\infty} \frac{\Omega(k,m)}{(k+1)!} \, = (-1)^m m! \sum_{k=1}^{\infty} \frac{(-1)^{k+1}}{(k+1)} \!\!
\sum_{n=m+1}^{\infty}\!\! \frac{\big|S_1(n,m+1)\big|}{n!\,n^{k}}
\end{align*}
which is identical with \eqref{g8t6v877} if setting $m=p$. Furthermore, it is well known that 
\be\notag
\frac{\big|S_1(n+1,m+1)\big|}{n!}\, =\,  P_{m}\big(H_{n}^{(1)}, -H_{n}^{(2)},\ldots, 
(-1)^{m-1}H_{n}^{(m)}\big)\,. 
\ee
see \cite[p.~217]{comtet_01}, \cite[p.~1395]{shen_01}, \cite[p.~425, Eq.~(43)]{kowalenko_01}, \cite[Eq.~(16)]{iaroslav_09},
which immediately gives \eqref{g8t6v877b} and completes the proof.

\subsection*{Proof of formula \eqref{kq9jdn2dd}}
This formula straightforwardly follows form the fact that $\kappa_p=F_p(1)$, see \cite[p.~307, 318 \emph{et seq.}]{candelpergher_01},
where $F_{p}(s)$ is the special function introduced in the above--cited reference.

\bibliographystyle{crelle}

\begin{thebibliography}{10}

\providecommand{\url}[1]{\texttt{#1}}
\providecommand{\urlprefix}{URL }
\expandafter\ifx\csname urlstyle\endcsname\relax
  \providecommand{\doi}[1]{doi:\discretionary{}{}{}#1}\else
  \providecommand{\doi}{doi:\discretionary{}{}{}\begingroup
  \urlstyle{rm}\Url}\fi

\bibitem{stieltjes_01}
Correspondance d'Hermite et de Stieltjes. Vol.~1 and 2, Gauthier-Villars,
  Paris, 1905.

\bibitem{ramanujan_01}
Collected papers of Srinivasa Ramanujan, Cambridge, 1927.

\bibitem{apostol_02}
\textit{T.~M. Apostol}, Formulas for higher derivatives of the \text{Riemann}
  zeta function, Mathematics of Computation, vol.\ 44, no.\ 169, pp.\ 223--232
  (1985).

\bibitem{berndt_03}
\textit{B.~C. Berndt}, Ramanujan's Notebooks, Part I. Springer-Verlag, 1985.

\bibitem{berndt_02}
\textit{B.~C. Berndt}, On the \text{Hurwitz Zeta--function}, Rocky Mountain
  Journal of Mathematics, vol.~2, no.~1, pp.~151--157  (1972).

\bibitem{iaroslav_09}
\textit{\text{Ia.~V.} Blagouchine}, Expansions of generalized \text{Euler's}
  constants into the series of polynomials in $\pi^{-2}$ and into the formal
  enveloping series with rational coefficients only, Journal of Number Theory
  (Elsevier), vol.~158, pp.~365--396 and vol.~173, pp.~631--632,
  arXiv:1501.00740  (2016).

\bibitem{iaroslav_10}
\textit{\text{Ia.~V.} Blagouchine}, Three notes on \text{Ser's} and
  \text{Hasse's} representations for the zeta-functions, arXiv:1606.02044
  (2016).

\bibitem{iaroslav_08}
\textit{\text{Ia.~V.} Blagouchine}, Two series expansions for the logarithm of
  the gamma function involving \text{Stirling} numbers and containing only
  rational coefficients for certain arguments related to $\pi^{-1}$, Journal of
  Mathematical Analysis and Applications (Elsevier), vol.~442, pp.~404--434,
  arXiv:1408.3902  (2016).

\bibitem{iaroslav_11}
\textit{\text{Ia.~V.} Blagouchine}, A note on some recent results for the
  \text{Bernoulli} numbers of the second kind, J.\ Integer Seq., vol.\ 20, no.\
  3, Article 17.3.8, pp.\ 1-7, 2017, arXiv:1612.03292  (2017).

\bibitem{briggs_01}
\textit{W.~E. Briggs}, The irrationality of $\gamma$ or of sets of similar
  constants, Vid. Selsk. Forh. (Trondheim), vol.~34, pp.~25--28  (1961).

\bibitem{briggs_02}
\textit{W.~E. Briggs} and \textit{S.~Chowla}, The power series coefficients of
  $\zeta(s)$, The American Mathematical Monthly, vol.~62, pp.~323--325  (1955).

\bibitem{candelpergher_04}
\textit{B.~Candelpergher}, Ramanujan summation of divergent series, Lecture Notes in Mathematics series, Vol. 2185, Springer,  2017, hal-01150208.

\bibitem{candelpergher_01}
\textit{B.~Candelpergher} and \textit{M.-A. Coppo}, A new class of identities
  involving \text{Cauchy} numbers, harmonic numbers and zeta values, The
  Ramanujan Journal, vol.~27, pp.~305--328  (2012).

\bibitem{candelpergher_05}
\textit{B.~Candelpergher} and \textit{M.-A. Coppo}, Le produit harmonique des
  suites, Enseign.~Math., vol.~59, pp.~39--72  (2013).

\bibitem{candelpergher_02}
\textit{B.~Candelpergher}, \textit{M.-A. Coppo} and \textit{E.~Delabaere}, La
  sommation de \text{Ramanujan}, Enseign.~Math., vol.~43, pp.~93--132  (1997).

\bibitem{candelpergher_03}
\textit{B.~Candelpergher}, \textit{H.~Gadiyar} and \textit{R.~Padma}, Ramanujan
  summation and the exponential generating function $\sum_{k=0}^{\infty}
  \frac{z^k}{k!} \zeta'(-k)$, The Ramanujan Journal, vol.~21, pp.~99--122
  (2010).

\bibitem{coffey_07}
\textit{M.~W. Coffey}, Certain logarithmic integrals, including solution of
  monthly problem \#tbd, zeta values, and expressions for the \text{Stieltjes}
  constants, arXiv:1201.3393v1  (2012).

\bibitem{comtet_01}
\textit{L.~Comtet}, Advanced Combinatorics. The art of Finite and Infinite
  Expansions (revised and enlarged edition), D. Reidel Publishing Company,
  Dordrecht, Holland, 1974.

\bibitem{connon_08}
\textit{D.~F. Connon}, Some possible approaches to the \text{Riemann}
  hypothesis via the \text{Li/Keiper} constants, arXiv:1002.3484  (2010).

\bibitem{coppo_01}
\textit{M.-A. Coppo}, Nouvelles expressions des constantes de \text{Stieltjes},
  Expositiones Mathematic\ae, vol.~17, pp.~349--358  (1999).

\bibitem{coppo 02}
\textit{M.-A. Coppo} and \textit{P.~T. Young}, On shifted  \text{Mascheroni} series and hyperharmonic numbers, Journal of Number Theory, vol.~169, pp.~1--20  (2016).


\bibitem{dilcher_01}
\textit{K.~Dilcher}, Generalized \text{Euler} constants for arithmetical
  progressions, Mathematics of Computation, vol.~59 , pp.~259--282  (1992).

\bibitem{finch_01}
\textit{S.~R. Finch}, Mathematical constants, Cambridge University Press, USA,
  2003.

\bibitem{franel_01}
\textit{J.~Franel}, Note \no 245, L'Intermédiaire des mathématiciens, tome
  II, pp.~153--154  (1895).

\bibitem{gram_01}
\textit{J.~P. Gram}, Note sur le calcul de la fonction $\zeta(s)$ de
  \text{Riemann}, Oversigt. K. Danske Vidensk. (Selskab Forhandlingar),
  pp.~305--308  (1895).

\bibitem{hardy_03}
\textit{G.~H. Hardy}, Note on \text{Dr.~Vacca's} series for $\gamma$, The
  Quarterly journal of pure and applied mathematics, vol.~43, pp.~215--216
  (1912).

\bibitem{israilov_01}
\textit{M.~I. Israilov}, On the \text{Laurent} decomposition of
  \text{Riemann's} zeta function [in \text{Russian}], Trudy Mat. Inst. Akad.
  Nauk. SSSR, vol.~158, pp.~98--103  (1981).

\bibitem{jensen_02}
\textit{J.~L. W.~V. Jensen}, Sur la fonction $\zeta(s)$ de \text{Riemann},
Comptes-rendus de l'Académie des sciences, tome 104, pp.~1156--1159  (1887).

\bibitem{jensen_03}
\textit{J.~L. W.~V. Jensen}, Note \no 245. \text{Deuxième réponse}.
 \text{Remarques} relatives aux réponses de \\ 
\text{MM. Franel et Kluyver},  L'Intermédiaire des mathématiciens, tome II, pp.~346--347  (1895).

\bibitem{kowalenko_01}
\textit{V.~Kowalenko}, Properties and applications of the reciprocal logarithm
  numbers, Acta Applicand\ae~Mathematic\ae, vol.~109, pp.~413--437  (2010).

\bibitem{lehmer_01}
\textit{D.~H. Lehmer}, \text{Euler} constants for arithmetical progressions,
  Acta Arithmetica, vol.~27, pp.~125--142  (1975).

\bibitem{lehmer_02}
\textit{D.~H. Lehmer}, The sum of like powers of the zeros of the
  \text{Riemann} zeta function, Mathematics of Computation, vol.~50, no.~181,
  pp.~265--273  (1988).

\bibitem{todd_01}
\textit{J.~J.~Y. Liang} and \textit{J.~Todd}, The \text{Stieltjes} constants,
  Journal of Research of the National Bureau of Standards---Mathematical
  Sciences, vol.~76B, nos.~3--4, pp.~161--178  (1972).

\bibitem{zhang_01}
\textit{Z.~Nan-You} and \textit{K.~S. Williams}, Some results on the
  generalized \text{Stieltjes} constant, Analysis, vol.~14, pp.~147--162
  (1994).

\bibitem{pilehrood_01}
\textit{T.~H. Pilehrood} and \textit{K.~H. Pilehrood}, Criteria for
  irrationality of generalized \text{Euler's} constant, Journal of Number
  Theory, vol.~108, pp.~169--185  (2004).

\bibitem{shen_01}
\textit{L.-C. Shen}, Remarks on some integrals and series involving the
  \text{Stirling} numbers and $\zeta(n)$, The Transactions of the American
  Mathematical Society, vol.~347, no.~4, pp.~1391--1399  (1995).

\bibitem{sitaramachandrarao_01}
\textit{R.~Sitaramachandrarao}, Maclaurin coefficients of the \text{Riemann}
  zeta function, Abstracts of Papers Presented to the American Mathematical
  Society, vol.~7, no.~4, p.~280, *86T-11-236  (1986).

\bibitem{tasaka_01}
\textit{T.~Tasaka}, Note on the generalized \text{Euler} constants,
  Mathematical Journal of Okayama University, vol.~36, pp.~29--34  (1994).

\bibitem{xia_01}
\textit{L.~Xia}, The \text{parameterized-Euler-constant} function
  $\gamma_a(z)$, Journal of Number Theory, vol.~133, no.~1, pp.~1--11  (2013).

\bibitem{xu_01}
\textit{C.~Xu}, \textit{Y.~Yan} and \textit{Z.~Shi}, Euler sums and integrals
  of polylogarithm functions, Journal of Number Theory, vol.~165, pp.~84--108
  (2016).

\bibitem{young_01}
\textit{P.~T. Young}, A 2-adic formula for \text{Bernoulli} numbers of the
  second kind and for the \text{Nörlund} numbers, Journal of Number Theory,
  vol.~128, pp.~2951--2962  (2008).

\bibitem{young 02}
\textit{P.~T. Young}, Rational series for multiple zeta and log gamma functions, Journal of Number Theory,
  vol.~133, pp.~3995--4009  (2013).


\end{thebibliography}

\end{document}